\documentclass{article}
\usepackage[english]{babel}
\usepackage{amsmath,amssymb,bbm,latexsym,theorem}

\newcommand{\nocomma}{}
\newcommand{\tmname}[1]{\textsc{#1}}
\newcommand{\tmop}[1]{\ensuremath{\operatorname{#1}}}
\newenvironment{proof}{\noindent\textbf{Proof\ }}{\hspace*{\fill}$\Box$\medskip}
\newtheorem{corollary}{Corollary}
\newtheorem{lemma}{Lemma}
\newtheorem{proposition}{Proposition}
{\theorembodyfont{\rmfamily}\newtheorem{remark}{Remark}}
\newtheorem{theorem}{Theorem}
\begin{document}

\title{Topological Symmetries of $\mathbbm{R}^3$, II}

\author{Fang Sun}

\maketitle

This paper should be viewed as an extension of [KS]. It deals with orientation
reversing topological actions of finite groups $G$ on the Euclidean space
$\mathbbm{R}^3$. All actions in this paper are assumed to be effective.

\begin{remark}
  This paper combined with [KS] gives complete classification of finite
  topological symmetries of $\mathbbm{R}^3$, namely:
\end{remark}

\begin{theorem}
  Let $G$ be a finite group acting topologically on $\mathbbm{R}^3$, then $G$
  is isomorphic to a finite subgroup of $O (3)$.
\end{theorem}

The structure of this paper is very similiar to that of [KS]. We start with
some preliminary result and then define six types of Obstruction Kernels, from
Type A to Type F.

Given the classification of finite groups acting (topologically) orientation
presercingly on $\mathbbm{R}^3$, it is clear the group $G$ is an extension of
such group by $\mathbbm{Z}_2$. The Obstruction Kernels will allow us to
exclude all possibilities where the group $G$ is not contained in $O (3)$.

\

\section{Preliminaries}

\begin{lemma}
  If $f$ is an orientation reversing involution of $S^3$ with fixed point set
  $S^2$, then $f$ permutes the two components of $S^3 - S^2$.
\end{lemma}

\begin{proof}
  Suppose the contrary horlds. Let $S^3 - S^2$ be $A \cup B$ where $A, B$ are
  the two components. $f$ restricts to a homeomorphism on $A$, i.e.,
  $\mathbbm{Z}_2$ acts on $A$. Since $A$ is acyclic, $A^{\mathbbm{Z}_2}$ is
  nonempty according to Smith Theory (cf. [B] p.145). This is impossible since
  $(S^3)^{\mathbbm{Z}_2} = S^2$.
\end{proof}

\begin{theorem}
  If $D_{2 n} (n > 2)$ acts on $\mathbbm{R}^3$ such that $\mathbbm{Z}_n
  \subseteq D_{2 n}$ is the collection of orientation preserving
  homeomorphisms, then $(S^3)^{D_{2 n}} \cong S^1$.
\end{theorem}

\begin{proof}
  Assume this is not the case. Then $(S^3)^{D_{2 n}} =
  [(S^3)^{\mathbbm{Z}_n}]^{\mathbbm{Z}_2} = (S^1)^{\mathbbm{Z}_2} = S^0$. Let
  $a, b$ be standard generators of $D_{2 n}$. For $0 \leqslant i \leqslant n -
  1$, $a^i b$ is an orientation reversing homeomorphism, so $(S^3)^{\langle
  a^i b \rangle} \cong S^0$ or $S^2$.
  
  {\underline{Case 1}}: If for all $i$, $(S^3)^{\langle a^u b \rangle} = S^0$,
  then the fixed point set of any nontrivial subgroup of $D_{2 n}$ is in $S^1
  = (S^3)^{\mathbbm{Z}_n}$. Thus $D_{2 n}$ acts on $S^3 - S^1$ freely. But
  this implies (since $S^3 - S^1$ is a homological 1-sphere) that $D_{2 n}$ is
  cyclic, which is impossible.
  
  {\underline{Case 2}}: If there exist $0 \leqslant i \leqslant n - 1$ such
  that $(S^3)^{\langle a^i b \rangle} \cong S^2$. Fix $i$ and denote
  $(S^3)^{\langle a^i b \rangle}$ as $S^2_0$. Now $a (S^3)^{\langle a^i b
  \rangle} = (S^3)^{a \langle a^i b \rangle a^{- 1}} = (S^3)^{\langle a^{i +
  2} b \rangle}$, denote this set as $S^2_1$(as the name suggests, it is
  homeomorphic to $S^2$). $S_1^2 \cap S_0^2 = (S^3)^{\langle a^i b, a^{i + 2}
  b \rangle} = (S^3)^{\langle a^2, a^i b \rangle} = [(S^3)^{\langle a^2
  \rangle}]^{\mathbbm{Z}_2} = (S^1)^{\mathbbm{Z}_2} = S^0$. Note that this
  $S^0$ is the fixed point set of $D_{2 n}$ and $S^0 \subseteq S^2_j$, $j = 0,
  1$.
  
  Since $\langle a \rangle$ is a cyclic group acting orientation preservingly,
  $(S^3)^{\langle a \rangle} \cong S^1$. $S^0 \subseteq S^1$, and $S^1 - S^0$
  has two components. Denote them as $\mathbbm{R}_+$ and $\mathbbm{R}_-$.
  
  {\underline{Claim}}: For $j = 0, 1$, $\mathbbm{R}_+$ and $\mathbbm{R}_-$
  lie in different components of $S^3 - S_j^2$.
  
  {\underline{Proof}}: Suppose otherwise. The homeomorphism $a^{i + 2 j} b$
  fixes $S^2_j$, and by the preceding lemma it permutes the two components. On
  the other hand, $D_{2 n} / \langle a \rangle$ acts on $(S^3)^{\langle a
  \rangle} = S^1$. In particular $a^{i + 2 j} b (\mathbbm{R}_+ \cup
  \mathbbm{R}_-) =\mathbbm{R}_+ \cup \mathbbm{R}_-$, but $\mathbbm{R}_+ \cup
  \mathbbm{R}_-$ is mapped to the other component of $S^3 - S^2_j$ which
  contains none of $\mathbbm{R}_+$ or $\mathbbm{R}_-$. This is a
  contradiction.
  
  For $j = 0, 1$, let $A_j, B_j$ be the components of $S^3 - S_j^2$ where
  $\mathbbm{R}_- \subseteq A_j, \mathbbm{R}_+ \subseteq B_j$. The intersection
  $S^1 \cap S_j^2 = (S^3)^{\langle a, a^{i + 2 j} b \rangle} = (S^3)^{D_{2 n}}
  = S^0$, whence $S_1^2 - S^0 \subseteq S^3 - S_0^2$. Now we must have either
  $S_1^2 - S^0 \subseteq A_0$ or $S_1^2 - S^0 \subseteq B_0$.
  
  If $S_1^2 - S^0 \subseteq B_0$, then $A_0 \cup S_0^2 \subseteq A_1 \cup B_1
  \cup S^0$ by taking complement. So $A_0 \subseteq A_1 \cup B_1$. Since $A_0$
  is connected, either $A_0 \subseteq A_1$ or $A_0 \subseteq B_1$. The latter
  is not possible because $\mathbbm{R}_- \nsubseteq B_1$. Thus $A_0 \subseteq
  A_1$. Obviously $A_0 \neq A_1$ and $A_0 \subset A_1$. By definition, $a A_0
  = A_1$ and $a^n A_0 = A_0$. Then $a^n A_0 = a^{n - 1} A_1 \supset a^{n - 1}
  A_0 \supset a^{n - 2} A_0 \supset \ldots \supset A_0$, a contradiction.
  
  This forces $S_1^2 - S^0 \subseteq A_0$ to be the case. But this is
  impossible by an analogous argument as the one above.
  
  Therefore all possibilities lead to contradictions and the initial
  assumption fails.
\end{proof}

\

\section{Obstruction Kernel}

\begin{proposition}
  (Obstruction Kernel of Type A) Let $G = (\mathbbm{Z}_p \oplus \mathbbm{Z}_q)
  \rtimes_{\varphi} \mathbbm{Z}_2$, where $\varphi (1)$ is multiplication by
  $1$(resp. $- 1$) on $\mathbbm{Z}_p$ (resp. $\mathbbm{Z}_q$). Then $G$ cannot
  act on $\mathbbm{R}^3$.
\end{proposition}

\begin{proof}
  Assume such action exists. Note that the subgroup $\mathbbm{Z}_q
  \rtimes_{\varphi} \mathbbm{Z}_2$ is a dihedral group $D_{2 q}$. If $G$ acts
  orientation preservingly, this is Obstruction Kernel of Type 0 in [KS], thus
  impossible. So we may assume the action to be not orientation preserving.
  
  The group $\mathbbm{Z}_p \oplus \mathbbm{Z}_q$ is the only subgroup of index
  2, whence it is the collection of orientation preserving homeomorphisms. In
  particular $\mathbbm{Z}_q$ acts orientation preservingly.
  
  Now $\mathbbm{Z}_p \rtimes_{\varphi} \mathbbm{Z}_2 =\mathbbm{Z}_p \oplus
  \mathbbm{Z}_2$. Therefore $(S^3)^{\mathbbm{Z}_p \oplus \mathbbm{Z}_2} =
  [(S^3)^{\mathbbm{Z}_2}]^{\mathbbm{Z}_p}$. Since the generator of
  $\mathbbm{Z}_2$ reverses orientation, $(S^3)^{\mathbbm{Z}_2} = S^2$ or
  $S^0$. In either case $[(S^3)^{\mathbbm{Z}_2}]^{\mathbbm{Z}_p} = S^0$. Now
  $(S^3)^{\mathbbm{Z}_p \oplus \mathbbm{Z}_2} =
  [(S^3)^{\mathbbm{Z}_p}]^{\mathbbm{Z}_2} = (S^1)^{\mathbbm{Z}_2}$. Note that
  $(S^3)^{\mathbbm{Z}_p} = (S^3)^{\mathbbm{Z}_p \oplus \mathbbm{Z}_q} =
  (S^3)^{\mathbbm{Z}_q} \cong S^1$, whence $(S^1)^{\mathbbm{Z}_2} = S^0 =
  [(S^3)^{\mathbbm{Z}_q}]^{\mathbbm{Z}_2} = (S^3)^{D_{2 q}}$. But by Theorem
  4, $(S^3)^{D_{2 q}} = S^1$, a contradiction.
\end{proof}

\begin{proposition}
  (Obstruction Kernel of Type B) Let $G = (\mathbbm{Z}_4 \oplus_{}
  \mathbbm{Z}_q) \rtimes_{\varphi} \mathbbm{Z}_2$, $q$ odd prime, $\varphi
  (1)$ is multiplication by $1$ (resp. $- 1$) on $\mathbbm{Z}_4$ (resp.
  $\mathbbm{Z}_q$). Then there is no action of $G$ on $\mathbbm{R}^3$ such
  that $\mathbbm{Z}_4 \oplus_{} \mathbbm{Z}_q$ is the collection of
  orientation preserving homeomorphisms.
\end{proposition}

\begin{proof}
  Assume such action exists. By Theorem 2, to produce a contradiction it
  suffices to prove that $(S^3)^{\mathbbm{Z}_q \rtimes_{\varphi}
  \mathbbm{Z}_2} = S^0$ ($\mathbbm{Z}_q \rtimes_{\varphi} \mathbbm{Z}_2$ is
  dihedral and $\mathbbm{Z}_q$ is the orientation preserving subgroup). We
  start with computations of the fixed point sets of various subgroups of $G$.
  
  Since $\mathbbm{Z}_4 \oplus \mathbbm{Z}_q$ is cyclic, we have
  $(S^3)^{\mathbbm{Z}_4 \oplus \mathbbm{Z}_q} = (S^3)^{\mathbbm{Z}_q} \cong
  S^1$.
  
  The fixed point set $(S^3)^{\mathbbm{Z}_q \rtimes_{\varphi} \mathbbm{Z}_2}
  = [(S^3)^{\mathbbm{Z}_q}]^{\mathbbm{Z}_2} = (S^1)^{\mathbbm{Z}_2}$ is either
  $S^1$ or $S^0$.
  
  The standard copy of $\mathbbm{Z}_2$ acts orientation reversingly, whence
  $(S^3)^{\mathbbm{Z}_2} = S^2$ or $S^0$.
  
  {\underline{Case 1}}: If $(S^3)^{\mathbbm{Z}_2} = S^0$, then this forces
  $(S^3)^{\mathbbm{Z}_q \rtimes_{\varphi} \mathbbm{Z}_2} =
  (S^1)^{\mathbbm{Z}_2}$ to be $S^0$ and we obtain the contradiction we are
  looking for.
  
  {\underline{Case 2}}: If $(S^3)^{\mathbbm{Z}_2} = S^2$, then
  $(S^3)^{\mathbbm{Z}_4 \rtimes_{\varphi} \mathbbm{Z}_2} =
  (S^3)^{\mathbbm{Z}_4 \oplus \mathbbm{Z}_2} =
  [(S^3)^{\mathbbm{Z}_2}]^{\mathbbm{Z}_4} = (S^2)^{\mathbbm{Z}_4}$. According
  to [E], any action on $S^2$ is conjugate to an orthogonal one. Thus
  $(S^2)^{\mathbbm{Z}_4}$ is either $S^0$ or empty. On the other hand, we have
  computed that $(S^3)^{\mathbbm{Z}_q \rtimes_{\varphi} \mathbbm{Z}_2}$ is
  either $S^1$ or $S^0$, Combining the two results, $(S^3)^{\mathbbm{Z}_q
  \rtimes_{\varphi} \mathbbm{Z}_2} = S^0$.
\end{proof}

\begin{proposition}
  (Obstruction Kernel of Type C) Let $G =\mathbbm{Z}_p \rtimes_{\varphi}
  \mathbbm{Z}_{2^{k + 1}}, k \geqslant 1$, $p$ odd prime, $\varphi (1)$ is
  multiplication by $- 1$. Then $G$ cannot act on $\mathbbm{R}^3$.
\end{proposition}

\begin{proof}
  Assume such action exists. If the action is orientation preserving, then $G$
  is Obstruction Kernel of Type 2 as in [KS], which is impossible. So it to
  consider the case where the action is not orientation preserving.
  
  It is not hard to see that $\mathbbm{Z}_p \rtimes_{\varphi}
  \mathbbm{Z}_{2^k} \cong \mathbbm{Z}_{2^k p}$ is the only subgroup of $G$
  with index 2. Thus this subgroup has to be the collection of homeomorphisms
  preserving orientation.
  
  Let $b = (0, 1) \in \mathbbm{Z}_p \rtimes_{\varphi} \mathbbm{Z}_{2^{k +
  1}}$. Then $b$ reverses orientation of $\mathbbm{R}^3$, whence
  $(S^3)^{\langle b \rangle} \cong S^0$ or $(S^3)^{\langle b \rangle} \cong
  S^2$.
  
  {\underline{Case 1}}: $(S^3)^{\langle b \rangle} \cong S^2$. In this case
  $(S^3)^{\langle b^2 \rangle} \supseteq (S^3)^{\langle b \rangle} = S^2$. But
  $b^2$ is orientation preserving and thus $(S^3)^{\langle b^2 \rangle} \cong
  S^1$, a contradiction. So this case is not possible.
  
  {\underline{Case 2}}: $(S^3)^{\langle b \rangle} \cong S^0$. In this case
  $(S^3)^G = S^0$. Since $\mathbbm{Z}_p \rtimes_{\varphi} \mathbbm{Z}_{2^k}
  \cong \mathbbm{Z}_{2^k p}$ acts orientation preservingly,
  $(S^3)^{\mathbbm{Z}_{2^k p}} \cong S^1$. The quotient $G /\mathbbm{Z}_{2^k
  p} \cong \mathbbm{Z}_2$ acts on this copy of $S^1$, so $G =\mathbbm{Z}_p
  \rtimes_{\varphi} \mathbbm{Z}_{2^{k + 1}}$ acts on the complement $S^3 -
  S^1$. The fixed point set of any nontrivial subgroup $H$ of $G$ is either
  $S^1$(when $\subseteq \mathbbm{Z}_{2^k p}$) or $S^0$(when $H \nsubseteq
  \mathbbm{Z}_{\circ^k p}$). Therefore the restriction to $S^3 - S^1$ is free.
  As before, this implies that $G$ is cyclic, which is obviously not the case.
  
  Thus either case leads to a contradiction and the proposition is proven.
\end{proof}

\begin{proposition}
  (Obstruction Kernel of Type D) Let $G =\mathbbm{Z}_p \rtimes_{\varphi}
  \mathbbm{Z}_2$, $p$ odd prime, $p \equiv 1 (\tmop{mod} 4)$, $\varphi (1) = n
  \in \mathbbm{Z}_p^{\ast}$ where $n^2 = - 1 (\tmop{mod} p)$. Then $G$ cannot
  act on $\mathbbm{R}^3$.
\end{proposition}

\begin{proof}
  The proof of Obstruction Kernel of Type 5 in [KS] carries verbatim.
\end{proof}

\begin{proposition}
  (Obstruction Kernel of Type E) Let $G = (\mathbbm{Z}_p \oplus \mathbbm{Z}_4)
  \rtimes_{\varphi} \mathbbm{Z}_2$, $p$ odd prime, $\varphi (1)$ is
  multiplication by $1$(resp. $- 1$) on $\mathbbm{Z}_p$(resp.
  $\mathbbm{Z}_4$), then $G$ cannot act on $\mathbbm{R}^3$ such that the
  collection os orientation preserving homeomorphisms is $\mathbbm{Z}_p \oplus
  \mathbbm{Z}_4$.
\end{proposition}

\begin{proof}
  Assume such action exists. Note that $\mathbbm{Z}_4 \rtimes_{\varphi}
  \mathbbm{Z}_2 \cong D_8$. To produce a contradiction, it suffices (by Theorem
  2) to prove $(S^3)^{D_8} \cong S^0$.
  
  By our assumption that the action restricted to $\mathbbm{Z}_p \oplus
  \mathbbm{Z}_4 \cong \mathbbm{Z}_{4 p}$ is orientation preserving, we have
  $(S^3)^{\mathbbm{Z}_p} = (S^3)^{\mathbbm{Z}_{4 p}} = (S^3)^{\mathbbm{Z}_4}
  \cong S^1$.
  
  The generator of $\mathbbm{Z}_2$ reverses orientation, whence
  $(S^3)^{\mathbbm{Z}_2} \cong S^2$ or $S^0$.
  
  If $(S^3)^{\mathbbm{Z}_2} = S^0$, then $(S^3)^{D_8} =
  [(S^3)^{\mathbbm{Z}_4}]^{\mathbbm{Z}_2} = (S^1)^{\mathbbm{Z}_2} = S^0$, and
  we are done.
  
  If $(S^3)^{\mathbbm{Z}_2} = S^2$, consider the subgroup $\mathbbm{Z}_p
  \rtimes_{\varphi} \mathbbm{Z}_2 \subseteq G$. By definition it is isomorphic
  to $\mathbbm{Z}_p \oplus \mathbbm{Z}_2$. And $(S^3)^{\mathbbm{Z}_p
  \rtimes_{\varphi} \mathbbm{Z}_2} = (S^3)^{\mathbbm{Z}_p \oplus
  \mathbbm{Z}_2} = [(S^3)^{\mathbbm{Z}_2}]^{\mathbbm{Z}_p} =
  (S^2)^{\mathbbm{Z}_p} = S^0$ since any action on $S^2$ is conjugate to an
  orthogonal one. On the other hand, $(S^3)^{\mathbbm{Z}_p \rtimes_{\varphi}
  \mathbbm{Z}_2} = [(S^3)^{\mathbbm{Z}_p}]^{\mathbbm{Z}_2} =
  (S^1)^{\mathbbm{Z}_2}$. Thus $(S^1)^{\mathbbm{Z}_2} = S^0$. Since
  $(S^1)^{\mathbbm{Z}_2}$ is also $[(S^3)^{\mathbbm{Z}_4}]^{\mathbbm{Z}_2} =
  (S^3)^{D_8}$, we obtain $(S^3)^{D_8} = S^0$.
\end{proof}

\begin{proposition}
  (Obstruction Kernel of Type F) $G = Q_{4 m}$ (the generalized quaternion
  group) cannot act on $\mathbbm{R}^3$.
\end{proposition}

\begin{proof}
  The proof follows verbatim from the proof of Obstruction Kernel of Type 3
  (c.f. [KS]), using the remark following Theorem 2 in the same paper and
  Obstruction Kernel of Type C.
\end{proof}

\

\section{The Cyclic Case}

In [KS], we considered extensions of (finite) subgroups of $\tmop{SO} (3)$ by
$\mathbbm{Z}_p$, $p$ prime. In the following sections, the algebraic aspects
of the situations are almost the same as their counterparts in the previous
paper. That paper has given an algebraic description to the possible results
of extensions. Thus we will not repeat the algebra part of those proofs, but
to filter them with the new obstruction kernels.

\begin{theorem}
  If $G$ acts on $\mathbbm{R}^3$ such that orientation preserving subgroup is
  cyclic, then $G$ is isomorphic to a subgroup of $O (3)$.
\end{theorem}

\begin{proof}
  It suffice to consider the case where the orientation preserving subgroup is
  of index 2.
  
  There is an short exact sequence{\tmname{}}
  \[ 0 \longrightarrow \mathbbm{Z}_n \longrightarrow G \longrightarrow
     \mathbbm{Z}_2 \longrightarrow 0 \]
  where $\mathbbm{Z}_n$ is the subgroup of orientation preserving
  homeomorphisms in $G$.
  
  The algebraic possibilities are known (c.f. [KS] Proposition 3,4,5). We will
  investigate then in a way analogous to [KS]. Let $n = 2^k m$, $m$ odd. Let
  $\varphi$ be the induced action of $\mathbbm{Z}_2$ on $\mathbbm{Z}_n$. Let
  $m = P \cdot Q$ where $\varphi (1)$ restrict to a multiplication by $+
  1$(resp. $- 1$) on $\mathbbm{Z}_P$(resp. $\mathbbm{Z}_Q$) (c.f. [KS])
  
  {\underline{Case 1}}: $k = 0$($n$ is odd).
  
  As in Proposition 3 of [KS], $G$ is either cyclic, dihedral or contains an
  Obstruction Kernel of Type A. In the last case $G$ cannot act. So $G
  \subseteq O (3)$.
  
  \
  
  {\underline{Case 2}}: $k = 1$($n = 2 m$, $m$ odd)
  
  We proceed as in Proposition 4 in [KS].
  
  i)Split Case:
  
  If $Q = 1$, $G =\mathbbm{Z}_n \oplus \mathbbm{Z}_2$, which is a subgroup of
  $O (3)$.
  
  If $P = 1$, $G$ is dihedral, thus $G \subseteq O (3)$
  
  If neither $P$ nor $Q$ is $1$, $G$ contains an Obstruction Kernel of Type A,
  which is impossible.
  
  ii)Non-split Case:
  
  If $Q = 1$, then $G$ is cyclic and $G \subseteq O (3)$.
  
  If $Q > 1$, then $G$ contains an Obstruction Kernel of Type C, a
  contradiction.
  
  In sum, for $n = 2 m$, $m$ odd, $G \subseteq O (3)$.
  
  \
  
  {\underline{Case 3}}: $k \geqslant 2 \nocomma$
  
  We argue as in Proposition 5 of [KS]. There are four possibilities for
  $\varphi$ restricted on the standard copy of $\mathbbm{Z}_{2^k}$ in
  $\mathbbm{Z}_n$.
  
  i)$\varphi (1)$ is multiplication by 1.
  
  Split Case: $G =\mathbbm{Z}_{2^k m} \rtimes_{\varphi} \mathbbm{Z}_2$. Either
  $P$ or $Q$ since otherwise there will be an Obstruction Kernel of Type A in
  $G$.
  
  If $Q = 1$, then $G =\mathbbm{Z}_n \oplus \mathbbm{Z}_2$, whence $G
  \subseteq O (3)$.
  
  If $P = 1$, then $G = (\mathbbm{Z}_{2^k} \oplus \mathbbm{Z}_Q)
  \rtimes_{\varphi} \mathbbm{Z}_2$, thus contains Obstruction Kernel of Type
  B. This is impossible.
  
  Non-split Case: In this case $G \cong (\mathbbm{Z}_P \oplus \mathbbm{Z}_Q)
  \rtimes_{\phi} \mathbbm{Z}_{2^{k + 1}}$ where $\phi (1)$ is multiplication
  by 1(resp. $- 1$) on $\mathbbm{Z}_P$(resp. $\mathbbm{Z}_Q$)
  
  If $Q = 1$, $G$ is cyclic thus isomorphic to a subgroup of $O (3)$.
  
  If $Q > 1$, $G$ contains an Obstruction Kernel of Type C.
  
  \
  
  ii)$\varphi (1)$ is multiplication by $- 1$
  
  Split Case: $G \cong \mathbbm{Z}_n \rtimes_{\varphi} \mathbbm{Z}_2$. Again
  either $P = 1$ or $Q = 1$ since an Obstruction Kernel of Type A will show up
  otherwise.
  
  If $Q = 1$, then $G$ contains an Obstruction Kernel of Type E, therefore
  this case is excluded.
  
  If $P = 1$, $G$ is dihedral.
  
  Non-Split Case: In such case $G \cong (\mathbbm{Z}_P \oplus \mathbbm{Z}_Q)
  \rtimes_{\phi} Q_{4 m}$ where for some $\phi$, but $Q_{4 m}$ is Obstruction
  Kernel of Type F, which implies this case cannot occur.
  
  \
  
  iii)$\varphi (1)$ is multiplication by $2^{k - 1} + 1$, then $G$ contains
  $\mathbbm{Z}_{2^k} \rtimes_{\varphi} \mathbbm{Z}_2$, a 2-group. This however
  contradicts the remark following Theorem 2 of [KS]. Thus this case is
  impossible.
  
  \
  
  iv)$\varphi (1)$ is multiplication by $2^{k - 1} - 1$. A same argument as
  above produces a contradiction.
  
  \
  
  In all the possible cases above, $G$ has to be isomorphic to a subgroup of
  $O (3)$.
\end{proof}

\begin{remark}
  The proof actually shows that $G$ is either $\mathbbm{Z}_{2 n}$,
  $\mathbbm{Z}_n \oplus \mathbbm{Z}_2$ or $D_{2 n}$.
\end{remark}

\

\section{The Dihedral Case}

The proof of the dihedral case follows the spirit of Proposition 8 and 9 in
[KS].

\begin{theorem}
  If $G$ acts on $\mathbbm{R}^3$ such that the subgroup of orientation
  preserving homeomorphisms is $D_{2 n}$, $n$ odd, $n \geqslant 3$, then $G$
  is isomorphic to a subgroup of $O (3)$.
\end{theorem}

\begin{proof}
  It suffice to consider the case where the action is not orientation
  preserving. The first half of the proof of Proposition 8 of [KS] carries
  verbatim. We have two cases(notations are borrowed from that proof):
  
  {\underline{Case 1}}: If $(2 a_1, 2 a_2, \ldots, 2 a_n) = (0, 0, \ldots, 0)$
  
  In this case the short exact sequence
  \[ 0 \longrightarrow D_{2 n} \longrightarrow G \longrightarrow
     \mathbbm{Z}_2 \longrightarrow 0 \]
  splits and thus $G \cong D_{2 n} \rtimes_{\varphi} \mathbbm{Z}_2$, $\varphi$
  as defined in Proposition 8 of [KS].
  
  We have $b_i = \pm 1$ for all $i$. There are three subcases:
  
  i) Both $\pm 1$ appears. In this case $G$ contains an Obstruction Kernel of
  Type A, which is impossible.
  
  ii) $b_i = 1$ for all $i$. then the action $\varphi$ is trivial on the
  cyclic subgroup $\mathbbm{Z}_n$. $\varphi$ is always trivial on the period 2
  generator $b$ of $D_{2 n}$ by definition. Thus $\varphi$ is trivial and $G
  \cong D_{2 n} \times \mathbbm{Z}_2$, a subgroup of $O (3)$.
  
  iii) $b_i = - 1$ for all $i$, in this case $G$ has been computed to be
  dihedral.
  
  \
  
  {\underline{Case 2}}: if $(2 a_1, 2 a_2, \ldots, 2 a_n) = \left( \frac{p_1
  - 1}{2} p_1^{n_1 - 1}, \ldots, \frac{p_k - 1}{2} p_k^{n_k - 1} \right)$
  
  In this case $G$(as computed in [KS]) contains an Obstruction Kernel of Type
  D, which is a contradiction.
  
  Summing the above results, we see in all possible cases $G \subseteq O (3)$.
\end{proof}

\begin{proposition}
  Suppose $G$ acts on $\mathbbm{R}^3$ such that the subgroup of orientation
  preserving homeomorphisms is $D_{2 n}$, $n$ even. If $n > 2$, then the
  extension $0 \rightarrow D_{2 n} \rightarrow G \rightarrow \mathbbm{Z}_2
  \rightarrow 0$ has to split.
\end{proposition}

\begin{proof}
  Assume this is not the case.
  
  Take the Sylow 2-subgroup of $D_{2 n}$. It must be a copy of $D_{2^{l +
  1}}$. Let $P$ be the Sylow 2-subgroup of $G$ containing this $D_{2^{l +
  1}}$. As a two group, $P \subseteq O (3)$, whence it is either
  $\mathbbm{Z}_{2^{l + 2}}$, $\mathbbm{Z}_{2^{l + 1}} \oplus \mathbbm{Z}_2$,
  $D_{2^{l + 2}}$ or $D_{2^{l + 1}} \times \mathbbm{Z}_2$. Containing $D_{2^{l
  + 1}}$, $P$ is not cyclic. It cannot be $D_{2^{l + 2}}$ or $D_{2^{l + 1}}
  \times \mathbbm{Z}_2$ either since that would make the extension $0
  \rightarrow D_{2 n} \rightarrow G \rightarrow \mathbbm{Z}_2 \rightarrow 0$
  splits ($P - D_{2^{l + 1}}$ contains an element of order 2). Thus $P \cong
  \mathbbm{Z}_{2^{l + 1}} \oplus \mathbbm{Z}_2$. This group is cyclic, and the
  same has to be true for $D_{2^{l + 1}}$, whence $l = 1$. In other word, $n =
  2 m$, $m$ odd.
  
  We have seen that $\tmop{Aut} D_{2 n} \cong \mathbbm{Z}_n \rtimes
  \mathbbm{Z}_n^{\ast}$. It is not hard to compute that $\tmop{Inn} D_{2 n}
  =\mathbbm{Z}_m \rtimes \{ \pm 1 \} \subseteq \mathbbm{Z}_n \rtimes
  \mathbbm{Z}_n^{\ast}$ where the embedding $\mathbbm{Z}_m \subseteq
  \mathbbm{Z}_n$ is canonical. Thus $\tmop{Out} D_{2 n} \cong \mathbbm{Z}_n
  \rtimes \mathbbm{Z}_n^{\ast} /\mathbbm{Z}_m \rtimes \{ \pm 1 \}$.
  
  The exact sequence
  \[ 0 \longrightarrow D_{2 n} \longrightarrow G \longrightarrow \mathbbm{Z}_2
     \longrightarrow 0 \]
  induces an abstract kernel $\phi : \mathbbm{Z}_2 \rightarrow \tmop{Out} D_{2
  n}$. $\phi (1)$ is represented by any conjugation of an element of $G - D_{2
  n}$ on $D_{2 n}$.
  
  Consider the subgroup $\{ e, a^m, b, a^m b \} \subseteq D_{2 n}$ where $e$
  stands for identity. This is a copy of $D_4$. Let $P'$ be a Sylow 2-subgroup
  of $G$ containing $D_4$. $P' \cong \mathbbm{Z}_4 \oplus \mathbbm{Z}_2$. In
  particular, $P'$ is abelian. Take $x \in P' - D_4$. the conjugation of $x$
  on $D_4$ is trivial. Now $x \in G - D_{2 n}$. Let $(t, s) \in \mathbbm{Z}_n
  \rtimes \mathbbm{Z}_n^{\ast} = \tmop{Aut} D_{2 n}$ be the conjugation by
  $x$. This automorphism sent $b$ to $a^t b$. Thus $t = 0 \in \mathbbm{Z}_n$.
  So $\phi (1)$ can be represented by $(0, s)$. $\phi (1)^2 = 0$ implies $s^2
  = \pm 1 \in \mathbbm{Z}_n^{\ast}$.
  
  Now assume $m = p_1^{n_1} \ldots p_k^{n_k}$ is the prime decomposition, then
  \[ \mathbbm{Z}_n^{\ast} =\mathbbm{Z}_2^{\ast} \times
     \mathbbm{Z}_{p_1^{n_1}}^{\ast} \times \ldots \times
     \mathbbm{Z}_{p_k^{n_k}}^{\ast} \cong \mathbbm{Z}_{p_1^{n_1}}^{\ast}
     \times \ldots \times \mathbbm{Z}_{p_k^{n_k}}^{\ast} \]
  Let $(b_1, \ldots, b_k)$ be the element in the rightmost group above
  corresponding to $s$.
  
  Since $D_{2 n}$ is centerless for $n > 2$, then (as computed in Proposition
  8 of [KS]) each abstract kernel corresponds to one and only one extension.
  
  {\underline{Case 1}}: If $s^2 = 1$, then $b_i = \pm 1$. Consider the
  homomorphism
  \[ \varphi : \mathbbm{Z}_2 \longrightarrow \mathbbm{Z}_n \rtimes
     \mathbbm{Z}_n^{\ast} = \tmop{Aut} D_{2 n} \]
  where $\varphi (1) = (0, s)$. The extension
  \[ 0 \longrightarrow D_{2 n} \longrightarrow D_{2 n} \rtimes_{\varphi}
     \mathbbm{Z}_2 \longrightarrow \mathbbm{Z}_2 \longrightarrow 0 \]
  induces the abstract kernel $\phi$. By uniqueness this split extension is
  equivalent to $0 \longrightarrow D_{2 n} \longrightarrow G \longrightarrow
  \mathbbm{Z}_2 \longrightarrow 0$, contradicting to the non-splitting
  assumption.
  
  {\underline{Case 2}}: If $s^2 = - 1$, then $b_i = m_i$ where $m_i^2 \equiv -
  1 (\tmop{mod} p_i^{n_i})$. Consider
  \[ f : \mathbbm{Z}_4 \longrightarrow \mathbbm{Z}_n^{\ast} \cong
     \underset{i}{\Pi} \mathbbm{Z}_{p_i^{n_i}}^{\ast} \]
  where $f (1) = \underset{i}{\Pi} m_i$. The canonical subgroup $\mathbbm{Z}_n
  \rtimes_f \mathbbm{Z}_2 \subseteq \mathbbm{Z}_n \rtimes_f \mathbbm{Z}_4$ is
  dihedral since $m^2 \equiv - 1$, and
  \[ 0 \longrightarrow D_{2 n} \longrightarrow \mathbbm{Z}_n \rtimes_f
     \mathbbm{Z}_4 \longrightarrow \mathbbm{Z}_2 \longrightarrow 0 \]
  induces the abstract kernel $\phi$. By uniqueness $G \cong \mathbbm{Z}_n
  \rtimes_f \mathbbm{Z}_4$. This however contains an Obstruction Kernel of
  Type D, which is thus impossible.
  
  In sum, either case leads to a contradiction and the assumption fails.
\end{proof}

\begin{theorem}
  If $G$ acts on $\mathbbm{R}^3$ such that the subgroup of orientation
  preserving homeomorphisms is $D_{2 n}$, $n$ even, then $G$ is isomorphic to
  a subgroup of $O (3)$.
\end{theorem}

\begin{proof}
  If $n = 2$, then $G$ is a 2-group and the result is known. So it suffice to
  consider the case where $n > 2$.
  
  By the preceding proposition, the extension $0 \rightarrow D_{2 n}
  \rightarrow G \rightarrow \mathbbm{Z}_2 \rightarrow 0$ splits. Thus $G \cong
  D_{2 n} \rtimes_{\varphi} \mathbbm{Z}_2$ for some $\varphi : \mathbbm{Z}_2
  \longrightarrow \tmop{Aut} D_{2 n} \cong \mathbbm{Z}_n \rtimes
  \mathbbm{Z}_n^{\ast}$. Let $\varphi (1) = (t, s)$.
  
  Consider the subgroup $\mathbbm{Z}_n \rtimes_{\varphi} \mathbbm{Z}_2
  \subseteq G$. By Remark 2, this subgroup is isomorphic to either
  $\mathbbm{Z}_{2 n}$, $\mathbbm{Z}_n \oplus \mathbbm{Z}_2$ or $D_{2 n}$.
  Since $n$ is even, $\mathbbm{Z}_n \rtimes_{\varphi} \mathbbm{Z}_2$ cannot be
  cyclic. Thus there are two possibilities.
  
  {\underline{Case 1}}: $\mathbbm{Z}_n \rtimes_{\varphi} \mathbbm{Z}_2 \cong
  \mathbbm{Z}_n \oplus \mathbbm{Z}_2$
  
  In this case $s = 1$. $\varphi (1)^2 = (t, 1)^2 = (2 t, 1) = (0, 1)$. Thus
  $t = 0$ or $t = m$, $m = \frac{n}{2}$.
  
  i)If $t = 0$, then $\varphi (1) = (0, 1)$, which is the identity isomorphism
  of $D_{2 n}$, thus $G = D_{2 n} \times \mathbbm{Z}_2$.
  
  ii)If $t = m$, then $\varphi (1) = (m, 1)$. Consider the subgroup $\{ e,
  a^m, b, a^m b \}$ of $D_{2 n}$. This is isomorphic to $\mathbbm{Z}_2 \oplus
  \mathbbm{Z}_2$, and $\varphi (1)$ restrict to an isomorphism of this group.
  It is not hard to see the resulted $(\mathbbm{Z}_2 \oplus \mathbbm{Z}_2)
  \rtimes_{\varphi} \mathbbm{Z}_2$ is isomorphic to $Q_8$. But this is
  Obstruction Kernel of Type F, a contradiction.
  
  {\underline{Case 2}}: $\mathbbm{Z}_n \rtimes_{\varphi} \mathbbm{Z}_2 \cong
  D_{2 n}$
  
  In this case $s = - 1$. We divide the discussion by parity of $t$.
  
  i)If $t$ is odd. Then $G$ is dihedral as computed in Proposition 9 of [KS].
  
  ii)If $t$ is even. Let $r = \frac{t}{2}$. The element $(a^r b, 1)$ is of
  order 2, and $(a^r b, 1) \in G - D_{2 n}$. Thus $G \cong D_{2 n}
  \rtimes_{\phi} \mathbbm{Z}_2$ where $\phi (1)$ is the conjugation by $(a^r
  b, 1)$. An easy computation shows that
  \begin{eqnarray*}
    (a^r b, 1) \cdot (a, 0) \cdot (a^r b, 1) = (a, 0) &  & \\
    (a^r b, 1) \cdot (0, 1) \cdot (a^r b, 1) = (0, 1) &  & 
  \end{eqnarray*}
  Thus $\phi (1)$ is the identity and $G \cong D_{2 n} \times \mathbbm{Z}_2$.
  
  We have seen that in all possible cases, $G \nocomma \subseteq O (3)$.
\end{proof}

\begin{remark}
  The last case in the proof actually gives an alternative proof to
  Obstruction Kernel of Type 6 in [KS], since $G$ then contains a copy of
  $\mathbbm{Z}_2 \oplus \mathbbm{Z}_2 \oplus \mathbbm{Z}_2$, which is not
  allowed in the orientation preserving case.
\end{remark}

Summing up the above results, we have:

\begin{corollary}
  If $G$ acts on $\mathbbm{R}^3$ such that the subgroup of orientation
  preserving homeomorphisms is dihedral, then $G \subseteq O (3)$.
\end{corollary}

\section{The $A_4$ Case}

\begin{theorem}
  If $G$ acts on $\mathbbm{R}^3$ such that the subgroup of orientation
  preserving homeomorphisms is $A_4$, then $G \subseteq O (3)$.
\end{theorem}

\begin{proof}
  There are up to isomorphism 15 groups of order 24, among which only $A_4
  \times \mathbbm{Z}_2$ and $S_4$ contains $A_4$. Both are subgroup of $O
  (3)$.
\end{proof}

\section{The $S_4$ Case}

\begin{theorem}
  If $G$ acts on $\mathbbm{R}^3$ such that the subgroup of orientation
  preserving homeomorphisms is $S_4$, then $G \subseteq O (3)$.
\end{theorem}

\begin{proof}
  The only extension of $S_4$ by $\mathbbm{Z}_2$ is $S_4 \times
  \mathbbm{Z}_2$(cf. [KS]).
\end{proof}

\section{The $A_5$ Case}

\begin{theorem}
  If $G$ acts on $\mathbbm{R}^3$ such that the subgroup of orientation
  preserving homeomorphisms is $A_5$, then $G \subseteq O (3)$.
\end{theorem}

\begin{proof}
  There are only two extensions of $A_5$ by $\mathbbm{Z}_2$: $A_5 \times
  \mathbbm{Z}_2$ and $S_5$(cf. [KS]). It suffice to prove that $S_5$ cannot act
  on $\mathbbm{R}^3$. Now suppose there is such an action.
  
  There is a subgroup of $S_5$ isomorphic to $G A (1, 5) \cong \mathbbm{Z}_5
  \rtimes \mathbbm{Z}_5^{\ast}$. This group cannot be embedded in $A_5$, thus
  the action restrict to an orientation reversing one on it (alternatively one
  can use the proof in [KS] to show that this group cannot act orientation
  preservingly). There is only one subgroup in $\mathbbm{Z}_5 \rtimes
  \mathbbm{Z}_5^{\ast}$ on index 2 and it is a copy of $D_{10}$. This $D_{10}$
  then has to be the subgroup of orientation preserving homeomorphisms.
  
  By the dihedral case, $\mathbbm{Z}_5 \rtimes \mathbbm{Z}_5^{\ast}$ has to be
  a subgroup of $O (3)$. It is not abelian, thus has to be either $D_{20}$ or
  $D_{10} \times \mathbbm{Z}_2$. Neither can be the case(the former is
  discussed in [KS], while the latter can be done by comparing Sylow
  2-subgroups). This contradicts with the assumption.
\end{proof}

\
\
\begin{flushleft}
\textbf{{\Large Acknowledgments:}}
\end{flushleft}
The author owes greatly to the generous help of Prof. Kwasik.

\
\begin{center}
\textbf{{\Large References:}}
\end{center}
\begin{flushleft}
[B] G. Bredon, Introduction to Compact Transformation Groups. Pure and Applied Mathematics, Vol. 46. Academic Press, New York, London, 1972.

[E] S. Eilenberg, Sur les transformations p{\'e}riodiques de la surface de
sph{\`e}re, Fund. Math 22 (1934), 28-41.

[KS] S. Kwasik and F. Sun, Topological Symmetries of $\mathbb{R}^3$. Preprint, Tulane University, 2016
\end{flushleft}

\end{document}